\documentclass[lettersize,onecolumn]{IEEEtran}
\usepackage{amsmath,amsfonts}
\usepackage{array}
\usepackage[caption=false,font=normalsize,labelfont=sf,textfont=sf]{subfig}
\usepackage{textcomp}
\usepackage{stfloats}
\usepackage{url}
\usepackage{verbatim}
\usepackage{subfig}
\usepackage{graphicx}
\def\BibTeX{{\rm B\kern-.05em{\sc i\kern-.025em b}\kern-.08em
    T\kern-.1667em\lower.7ex\hbox{E}\kern-.125emX}}
\usepackage{balance}

\usepackage{multirow}
\usepackage[a4paper,margin=1in]{geometry}
\usepackage{amsmath,amssymb,amsthm,mathtools}
\usepackage{enumitem}
\usepackage{mathrsfs}
\usepackage[utf8]{inputenc}
\usepackage{algorithm}
\usepackage{algpseudocode}
\usepackage{cite}
\usepackage{titlesec}

\newtheorem{theorem}{Theorem}

\newtheorem{lemma}{Lemma}
\newtheorem{corollary}{Corollary}

\usepackage{graphicx} 
\usepackage{color}

\newtheorem{assumption}{Assumption}

\newcommand{\tx}{\tilde{x}}

\newcommand{\td}{\mathrm{d}}

\title{Operator splitting based diffusion samplers and improved convergence analysis}
\author{Peiyi Liu, Zhaoqiang Liu and Yiqi Gu
\thanks{P. Liu and Y. Gu are supported by National Natural Science Foundation of China Major Research Plan (92370101).

The authors are with the School of Mathematical Sciences (P. Liu, Z. Liu and Y. Gu) and the School of Computer Science and Engineering (Z. Liu), University of Electronic Science and Technology of China, Sichuan 611731, China (email: peiyiliu@std.uestc.edu.cn; yiqigu@uestc.edu.cn; zqliu12@gmail.com).}
}

\begin{document}

\maketitle

\begin{abstract}
In this paper, we develop a class of samplers for the diffusion model using the operator-splitting technique. The linear drift term and the nonlinear score-driven drift of the probability flow ordinary differential equation are split and applied by flow maps alternatively. Moreover, we conduct detailed analyses for the second-order sampler, establishing a non-asymptotic total variation distance error bound of order $O(d/T^2+\sqrt{d}\varepsilon_{\mathrm{score}}+d\varepsilon_{\mathrm{Jac}})$, where $d$ is the data dimension; $T$ is the number of sampling steps; $\varepsilon_{\mathrm{score}}$ and $\varepsilon_{\mathrm{Jac}}$ measure the discrepancy between the actual score function and learned score function. Our bound is sharper than existing works, yielding bounds of $O(d^p/T^2)$ with some $p>1$ for specific second-order samplers. Numerical experiments on a two-dimensional synthetic dataset corroborate the established quadratic dependence on the step size $1/T$ in the error bound.
\end{abstract}

\section{Introduction}
\IEEEPARstart{D}{iffusion} models have emerged as a central 
paradigm in modern machine learning, achieving state-of-the-art performance in image, audio and other modalities generation. Early works formulate generation as the reversal of a fixed Markov noising process and demonstrate that such models can match or surpass generative adversarial networks (GANs) in image quality while being substantially more stable to train \cite{Sohl2015,Ho2020,Dhariwal2021,Rombach2022}. Score-based generative modeling via stochastic differential equations (SDEs) further unifies discrete-time diffusion models with continuous-time stochastic sampling, and has led to a mature theoretical and algorithmic framework for diffusion-based learning and generation \cite{Song2021,Croitoru2023,Yang2023,Tang2025}.

A major bottleneck for practical deployment is sampling cost: high-fidelity generation typically requires hundreds or even thousands of function evaluations along the reverse-time trajectory. This has triggered a large body of work on training-free acceleration of diffusion sampling, either by designing specialised high-order ordinary differential equation (ODE) solvers for the probability flow ODE (PF-ODE) \cite{Song2021}, or by constructing carefully tuned discretizations of the SDE counterpart. On the ODE side, DPM-Solver and its multi-step variants provide tailored, high-order solvers that can generate high-quality samples within tens of neural function evaluations (NFEs) \cite{Lu2022,Zhang2023,Zhao2023,Xue2023}. In parallel, a complementary line of work pursues distillation or consistency-type models to amortise long trajectories into a few steps \cite{Salimans2022,Song2023}. These methods offer impressive empirical gains, but their design is often heuristic, and their convergence properties remain only partially understood, especially in high dimensions.

As a result, there has been a rapidly growing interest in rigorous convergence theory for diffusion models. A first wave of results establishes polynomial-time guarantees for DDPM or SDE-based samplers under various structural conditions on data distribution and score approximation error. Chen, Chewi et al. \cite{Chen2023_1} improve the dependence on the score error and weaken smoothness assumptions. These works typically control either the Wasserstein-2 distance or total variation (TV) distance between the generated and target distributions, and make explicit the trade-off between step size, score accuracy and data dimension. Related deterministic analyses for flow matching and ODE-type samplers \cite{Benton2024} further clarify how numerical integration error propagates to distributional error, but are not directly tailored to the specific PF-ODE used in diffusion models.

More recently, several papers have focused specifically on probability flow ODE samplers. Li et al. \cite{Li2023} develop non-asymptotic convergence guarantees for both DDPM and PF-ODE based samplers, and introduce accelerated variants that can reach $O(1/T^2)$ convergence in the number of steps $T$. Li et al. \cite{Li2024_1} provide a sharp convergence theory for PF-ODEs that refines the dependence on the score error and time horizon; Gao and Zhu \cite{Gao2024} obtain non-asymptotic convergence rates for a broad class of PF-ODE samplers in Wasserstein distances under log-concavity; and Huang et al. \cite{Huang2025} establish convergence guarantees under related structural conditions. Parallel to these PF-ODE-focused results, there is a growing literature on provably accelerating diffusion models via higher-order or adaptive schemes \cite{Li2024_2,Wu2024}, which often draw inspiration from classical high-order SDE or ODE integrators.

Among these, the works \cite{Li2024_3} and \cite{Li2025} are particularly relevant. Both works design training-free, accelerated samplers whose construction mirrors high-order numerical integrators, and derive non-asymptotic convergence guarantees for DDIM and PF-ODE-style methods under mild assumptions on the data distribution and score error. In particular, Li et al. prove the $O(d^6/T^2)$ error bound for a specific second-order sampler \cite{Li2024_3}. After that, they prove $O(\frac{d^{2+K}}{T^K})$ error bounds for a general type of $K$-th order samplers \cite{Li2025}. Therefore, for a target accuracy $\varepsilon$ in TV distance, the $K$-th order sampler requires on the order of $d^{1+2/K}\varepsilon^{-1/K}$ score evaluations. This work provides a general high-order framework, but the dependence on dimension $d$ is super-linear, and the analysis exploits a rather involved high-order Lagrange interpolation of the PF-ODE integral.

In this paper, we develop a new class of samplers for PF-ODEs based on operator splitting, along with a detailed TV distance analysis that yields improved dimension dependence under standard assumptions. The analysis framework differs markedly from existing PF-ODE sampler theories. Starting from the usual VP-type forward SDE, we view the PF-ODE as the sum of a linear drift and a nonlinear score-driven drift. Specifically, we consider the semi-linear ODE
$$\frac{\td x(t)}{\td t}=A(t)x(t)+B(t,x(t)),$$
with $A(t):[0,1]\xrightarrow{}\mathbb{R}$ and $B(t,x(t)): [0,1]\times \mathbb{R}^d\xrightarrow{}\mathbb{R}^d$. We then compose the two flow maps via operator splitting to obtain a full PF-ODE update. Let $S_B(\Delta t)$ denote the flow maps generated by $B$ over time step $\Delta t$. Then the operator splitting method is formally given by
\begin{equation*}
\Phi_h=e^{a_1\Delta tA}S_B(b_1\Delta t)e^{a_2\Delta tA}S_B(b_2\Delta t)\dots \dots e^{a_K\Delta tA}S_B(b_K\Delta t),
\end{equation*}
where $a_i,b_i$ are constants satisfying $\sum_{i=1}^K a_i=1$, $\sum_{i=1}^K b_i=1$ and $K\in \mathbb{N}^+$. Especially, this numerical method is called Strang splitting when $K=2$. This leads to a practically simple, training-free algorithm that only requires access to the learned score field. From a numerical analysis perspective, our scheme is closely related to classical operator splitting methods \cite{Marchuk1990,Yanenko1971,Hundsdorfer2013}, but it is tailored to the specific structure of the PF-ODE in the diffusion model.

Our main theoretical contribution is a global TV distance error bound for the proposed Strang-splitting sampler, under standard Lipschitz and linear-growth assumptions on the score and its Jacobian, and mild moment assumptions on data distribution. The bound has three components: A discretization term of order $O(\frac{d}{T^2})$; a term of order $O(\sqrt{d}\varepsilon_{\mathrm{score}})$ controlled by the $L^2$ score estimation error; and a Jacobian-related term that is linear in $d$. Consequently, our error bound of the sampler is sharper than the bound $O(\frac{d^4}{T^2})$ proved in \cite{Li2025}, which we reduce the dependence on the dimension $d$ from quartic to linear. To reach a prescribed TV accuracy $\varepsilon$ in the idealised setting of accurate scores, it suffices to choose $h\sim O(\sqrt{\varepsilon/d})$, implying that the number of NFEs scales as $O(\sqrt{d/\varepsilon})$. This is the best-known dependence on $\varepsilon$ for second-order schemes while yielding a significantly milder dependence on the dimension $d$ than the $d^2$ factor appearing in the second-order scheme of Li et al. \cite{Li2025}. At the same time, our analysis remains robust to inexact scores and requires no assumptions about the target distribution or structural low dimensionality.

Beyond establishing improved TV bounds, our proofs introduce several technical ingredients that may be of independent interest. We show that the discretization correction induced by operator splitting behaves like a ``reasonable" drift field, with controlled growth and divergence, which allows us to turn the formal $h^2$ local error into a genuine $O(h^2)$ contribution in TV distance at the distributional level. We also develop a TV-based stability estimate for PF-ODE flows under perturbations of both the drift and its Jacobian, which cleanly separates the impact of discretization from that of score approximation. Finally, we complement the theory with experiments on two-dimensional synthetic data, where we can compute the TV distance explicitly: These experiments confirm the predicted $O(h^2)$ scaling and demonstrate that our Strang-based sampler matches the theoretical behaviour even when the score is learned by a neural network of finite hidden layers and width.

Overall, this work contributes to the emerging numerical analysis theory of diffusion sampling: It connects classical operator-splitting schemes with modern PF-ODE samplers, provides dimension-aware TV bounds under widely used regularity assumptions, and offers a theoretically grounded alternative to more intricate high-order constructions. We hope that the techniques developed here will also be useful for analysing other composition type samplers and for designing further higher-order, yet theoretically tractable, numerical algorithms for diffusion generative modeling.

\section{Diffusion Probabilistic Models}
Here, we briefly recap the diffusion probabilistic models and the associated ODE we ultimately need to solve. A diffusion model starts from an unknown data distribution $q(x,0)$, the probability density function (PDF) of the true data. Here, the second variable of $q$ being zero means the initial time $t=0$. Let $x_0 \in\mathbb{R}^d$ be a $d$-dimensional sample from $q(x,0)$, conditional on which a forward noising process can be defined following a Gaussian
\begin{equation}\label{forward}
x_t|x_0\sim\mathcal{N}(\alpha(t)x_0,\sigma^2(t)I),
\end{equation}
where $\alpha(t),\sigma(t) \ge 0$ are differentiable ``noise schedules" and $t\in[0,1]$. Let $q(x,t)$ be the PDF of $x_t$. Note that $q(x,t)$ is differentiable in $x$ and $t$.


Kingma et al. \cite{Kingma2021} proved that the following stochastic differential
equation (SDE) has the same transition distribution as in \eqref{forward}:
\begin{equation}\label{01}
\td x_t=f(t)x_t\td t+g(t)\td w_t, \quad x_0\sim q(x,0),\quad t\in[0,1],
\end{equation}
where $w_t\in R^d$ is the standard Wiener process, and
\begin{equation*}
f(t)=\frac{\td(\log\alpha(t))}{\td t}, \quad
g^2(t)=\frac{\td \sigma^2(t)}{\td t}-2\frac{\td\log \alpha(t)}{\td t}\sigma^2(t).
\end{equation*}
Additionally, Song et al. \cite{Song2021} show that the forward SDE \eqref{01} admits a reverse-time dynamics that transports $x_1\sim q(x,1)$ back to the data distribution:
\begin{equation*}
\td x_t=\bigl(f(t)x_t-g^2(t)\nabla_x\log q(x,t)\bigr)\td t+g(t)d\bar{w_t}, \quad x_1\sim q(x,1),
\end{equation*}
where $\bar{w_t}$ is a reverse-time Wiener process and the function $\nabla_x\log q(x(t),t)$ is called the score function \cite{Song2021}. The inherent randomness of the Wiener process limits the step size during SDE discretization, often leading to non-convergence with large steps. To enable faster sampling, the probability flow ODE (PF-ODE), which preserves the same marginal distributions, can be used instead. Song et al. \cite{Song2021} proved that such an PF-ODE is formulated as
\begin{equation}\label{02}
\frac{\td x_t}{\td t}=f(t)x_t-\frac{1}{2}g^2(t)\nabla_{x}\log q(x,t),\quad x_1\sim q(x,1),
\end{equation}
where the marginal distribution of $x_t$ is exactly $q(x,t)$.

In practical diffusion models, one usually uses a neural network $\epsilon_{\theta}(x_t,t)$ to predict the Gaussian noise added to $x_t$ during the forward process, with $\epsilon_{\theta}(x_t,t)\approx-\sigma(t)\nabla_{x}\log q(x,t)$. Then \eqref{02} can be written as
\begin{equation}\label{PFODE}
\frac{\td x_t}{\td t}=h_\theta(x_t,t) := f(t)x_t + \frac{g^2(t)}{2\sigma(t)}\epsilon_{\theta}(x_t,t),\quad
x_1\sim \mathcal{N}(0,I).
\end{equation}
Samples from the diffusion model can be generated by solving the ODE backward from $t=1$ to $t=0$. 

Compared to stochastic SDE-based sampling, deterministic ODE sampling is both more numerically efficient and more amenable to analysis. The reverse-time probability-flow ODE yields a smooth, noise-free trajectory, allowing standard ODE solvers to take larger step sizes without sacrificing stability or sample quality, thus reducing the number of NFEs per sample. At the same time, the absence of stochastic perturbations makes trajectories easier to invert, interpolate, and manipulate, which is advantageous for tasks such as latent inversion, controlled editing, and theoretical analysis of the generative dynamics \cite{Lu2022}.

\section{Formulation of Samplers}\label{seC_2}
In this section, we propose a framework for constructing samplers of the diffusion model via operator splitting. We primarily discuss the formulations of the second, third, and fourth-order samplers. The framework can be easily generalized to any higher orders.

\subsection{Principle of Operator Splitting}
Operator splitting is a numerical strategy for time-dependent evolution equations in which a complex problem is decomposed into simpler subproblems that are advanced sequentially in time. Consider an abstract evolution equation with $u_0$ as the initial value,
\begin{equation}\label{04}
\frac{\td u}{\td t}=Au+B(u),\quad t\in[0,1],
\end{equation}
where $u:[0,1]\xrightarrow{}\mathbb{R}^d$; $A$ is a linear operator, and $B(\cdot)$ is a general (maybe nonlinear) operator.

Let $S_A(\Delta t)$ and $S_B(\Delta t)$ denote the flow maps generated by $A$ and $B$ over time step $\Delta t$, respectively. Namely, 
\begin{equation}\label{05}
\frac{\td u}{\td t}=Au
\end{equation}
and
\begin{equation}\label{06}
\frac{\td u}{\td t}=B(u)
\end{equation}
have solutions $u(\Delta t)=S_A(\Delta t)u_0=e^{\Delta t A}u_0$ and $u(\Delta t)=S_B(\Delta t)u_0$ whenever they are well-defined. 

Given $u(t^*)$, the solution evaluation from $t=t^*$ to $t=t^*+\Delta t$ can formally computed by 
\begin{equation*}
u(t^*+\Delta t)=S_{A+B}(\Delta t)u(t^*),
\end{equation*}
where $S_{A+B}$ denotes the flow maps of the full equation. In general, implementing the flow $S_{A+B}(\Delta t)$ directly is difficult. The core idea of operator splitting is to split $S_{A+B}$ into $S_A$ and $S_B$, which are implemented more simply. Common operator splitting schemes include\cite{Blanes2024}:\vspace{6pt}\\
(1) {\em Lie Splitting}\vspace{2pt}\\
In first-order Lie Splitting, one step of solution evaluation from $t=t^*$ to $t=t^*+\Delta t$ is approximated by
\begin{equation*}
u(t^*+\Delta t)\approx S_B(\Delta t)e^{\Delta t A}u(t^*),
\end{equation*}
which is a first-order scheme in $\Delta t$ under suitable regularity and bounds.\vspace{6pt}\\
(2) {\em Strang Splitting}\vspace{2pt}\\
To achieve second-order accuracy in time, one commonly uses Strang splitting. A single time step is defined by
\begin{equation*}
u(t^*+\Delta t)\approx e^{\frac{\Delta t} 2 A}S_B(\Delta t)e^{\frac{\Delta t} 2 A}u(t^*).
\end{equation*}
Under appropriate assumptions, this construction yields a second-order approximation in time while preserving the modular structure of the algorithm.

We remark that the time step $\Delta t$ can be either positive or negative. If $\Delta t>0$, the solution is computed for forward time marching. Otherwise, the solution is computed for backward time marching. 

\subsection{Samplers based on operator splitting} 
Under the framework of operator splitting, the PF-ODE~\eqref{PFODE} satisfies $A(t)=f(t)$, $B(t,x)=\frac{g^2(t)}{2\sigma(t)}\epsilon_{\theta}(x,t)$. Note that we need to solve \eqref{PFODE} backward from $t=1$ to $t=0$. For convenience, we introduce the following notation: $T\in \mathbb{N}^+$ is the number of time steps; $h=\frac{1}{T}$ is the step size; $t_n=nh$ ($n=0,\dots,T$) are equidistant time steps on $[0,1]$; $\tx_n:=x_{t_n}$ for $n=0,\dots,T$. 

\subsubsection{A second-order scheme}

We first propose a sampler such that the total variation (TV) distance between the distribution of the generated samples and the true data distribution decays at a second-order rate in $h$. For the PF-ODE, suppose $\tx_n$ is obtained, the second-order scheme employs Strang splitting, i.e.  
\begin{equation}\label{07}
\tx_{n-1}=e^{-\frac{h} 2 A}S_B(-h)e^{-\frac{h} 2 A}\tx_n.
\end{equation}

To match the second-order rate, one can adopt a second-order ODE solver to approximate the nonlinear flow $S_B(h)$. In this work, we employ the explicit second-order Runge-Kutta (RK2) \cite{Hairer1993} backward-in-time scheme. Therefore, the specific procedures for implementing \eqref{07} are as follows. First, we take a half time step along the linear flow $e^{\frac{h} 2 A}$, solving the linear ODE \eqref{05} from $t_n$ to $t_{n-\frac{1}{2}}$, with $\tx_n$ as the initial value, we have the solution
\begin{equation*}
\tx_n^*=\tx_n \exp \left(\int_{t_n}^{t_{n-\frac{1}{2}}}f(s)\td s\right).
\end{equation*}
Second, we take a full-time step along the nonlinear flow $S_B(h)$, The RK2 scheme to advance the solution of the nonlinear ODE \eqref{06} from $t_n$ to $t_{n-1}$, with $\tx_n^*$ as the initial value, gives the following scheme:
\begin{align*}
& k_1=-\frac{1}{2}\frac{g^2(t_n)}{\sigma(t_n)}\epsilon_\theta(\tx_n^*,t_n), \\
& k_2=-\frac{1}{2}\frac{g^2(t_{n-\frac{1}{2}})}{\sigma(t_{n-\frac{1}{2}})}\epsilon_\theta(\tx_n^*+\frac{h}{2}k_1,t_{n-\frac{1}{2}}), \\
& \tx_n^{**}= \tx_n^*+hk_2.
\end{align*}
Third, we take a half time step along the linear flow $e^{\frac{h}{2}A}$, solving the linear ODE \eqref{05} from $t_{n-\frac{1}{2}}$ to $t_{n-1}$, with $\tx_n^{**}$ as the initial value. Namely,
\begin{equation*}
\tx_{n-1}=\tx_n^{**}\exp \left(\int_{t_{n-\frac{1}{2}}}^{t_{n-1}}f(s)\td s\right).
\end{equation*}

To sum up, the second-order sampler scheme from $\tx_T$ to $\tx_0$ can be written as: 
\begin{equation}\label{sampler}
\begin{split}
& \tx_n^*=\tx_n \exp \left(\int_{t_n}^{t_{n-\frac{1}{2}}}f(s)\td s\right), \\
& k_1=-B(t_n,\tx_n^*), \\
& k_2=-B(t_{n-\frac{1}{2}},\tx_n^*+\frac{h}{2}k_1), \\
& \tx_n^{**}=\tx_n^*+hk_2,\\
& \tx_{n-1}=\tx_n^{**}\exp \left(\int_{t_{n-\frac{1}{2}}}^{t_{n-1}}f(s)\td s\right),
\end{split}
\end{equation}
for $n=T,T-1,\dots,1$.

\subsubsection{Higher-order schemes}
The framework proposed above can be extended to construct higher-order samplers. In practical applications, samplers of different orders can be selected based on the task's difficulty. Specifically, given $K\in\mathbb{Z}^+$, we can construct the $K$-th order sampler. First, we take a $K$-th order operator splitting scheme, i.e. 
\begin{equation*}
\tx_{n-1}=S_B(-b_Kh)e^{-a_K h A}\cdots S_B(-b_2 h t)e^{-a_2 h A} S_B(-b_1 h t)e^{-a_1 h A}\tx_n,
\end{equation*}
where $a_i,b_i$ are certain constants with $\sum a_i=\sum b_i=1$. Next, to implement the flow $S_B(\cdot)$, we choose a $K$-th order explicit Runge-Kutta process $\mathrm{RK}:\mathbb{R}^d \times [0,1]\times \mathbb{R}^+\xrightarrow{}\mathbb{R}^d$, $\mathrm{RK}_s(x,t,h)=x+h\sum_{i=1}^sb_ik_i$, with
\begin{align*}
& k_1=-B(t,x)\\
& k_2=-B(t+C_1h,x+a_{21}hk_1)\\
& \qquad \vdots\\
& k_s=-B(t+c_sh,x+a_{s1}hk_1+a_{s2}hk_2+\cdots+a_{s,s-1}hk_{s-1}),
\end{align*}
where $a_i,b_i,c_i$ are specific coefficients and $s$ is the number of stages. Then, the $K$-th order sampler is formulated as
\begin{align*}
&x^{(0)} = \tx_n,\\
&y^{(m)}=x^{(m)} \exp \left(\int_{t_{n-\sum_{i=1}^m a_i}}^{t_{n-\sum_{i=1}^{m+1} a_i}}f(s)\td s\right),\\
&x^{(m+1)}=\mathrm{RK}(y^{(m)},t_{n-\sum_{i=1}^m b_i},b_{m+1}h),\\
&m=0,1,\dots,s-1,\\
&\tx_{n-1} = x^{(s)},
\end{align*}
for $n=T,T-1,\dots,1$.  The construction of the above schemes is flexible. One can use different constants in high-order operator splitting schemes, and adopt ODE solvers other than Runge-Kutta methods for the nonlinear part \eqref{06}. 

\section{Error Analysis}
In this section, we conduct an error analysis of the second-order sampler introduced in Section III. Our goal is to show that, under mild regularity assumptions, the scheme has a global error of $O(h^2)$, where $h>0$ denotes the time step size. Throughout the following, $\|\cdot\|$ denotes the Euclidean $l_2$ norm on $\mathbb{R}^d$. We use the notation $E=O(h^k)$ to mean that $|E|\leq Ch^k$ for some constant $C>0$ and some power $k>0$ which are independent of $h$. When applied to vectors, $O(h^k)$ is understood measured in norm $\|\cdot\|$. 

\subsection{Setting and assumptions}
For simplicity of notations, we let $b(t,x(t)):=-\frac{1}{2}g^2(t)\nabla_x\log q(x,t)$. Then the ODE in \eqref{02} can be written as
\begin{equation}\label{08}
x'(t)=f(t)x(t)+b(t,x(t)),\quad t\in [0,1].
\end{equation}

We impose the following regularity and Lipschitz conditions.
\begin{assumption}\label{ass:01}
The scalar function $f,g\in C^2([0,1])$. The function $q(x,t)$ is twice continuously differentiable in $t$ and three times continuously differentiable in $x$ on $[0,1]\times\mathbb{R}^d$, with bounded derivatives up to the third order. Moreover, there exists a constant $L>0$ such that
\begin{align*}
\|b(t,x)-b(t,y)\|\le L\|x-y \|,
\end{align*}
for all $t\in[0,1]$ and $x,y\in \mathbb{R}^d$.
\end{assumption}

Note that the Lipschitzness of the score function in Assumption \ref{ass:01} is standard and has been used in prior works \cite{Lee2022}. 

As a first step, we recall the standard variation of constant representation for the solution of \eqref{08}, which we will repeatedly use to compare the exact evolution $x(\cdot)$ with the numerical flow. For any $s,t\in[0,1]$, we set 
\begin{align*}
m_f(t,s):=\exp \left(\int_s^t f(\tau) \td \tau  \right).
\end{align*}

\begin{lemma}\label{lemma:01}
Let $x(\cdot)$ solve \eqref{08}. For any $0\le t_{n-1}<t_n\le1$, 
\begin{equation}\label{09}
x(t_{n-1})=m_f(t_{n-1},t_n)x(t_n)+ \int_{t_n}^{t_{n-1}}m_f(t_{n-1},s)b(s,x(s))\td s.
\end{equation}
\end{lemma}

\begin{proof}
Let $\mu(t):=\exp \left(-\int_{t_n}^t a(\tau)d\tau\right)$, and thus $\mu'(t)=-f(t)\mu(t)$.
Multiplying \eqref{08} by $\mu(t)$ gives
\begin{align*}
\frac{d}{\td t}[\mu(t)x(t)]
&=\mu'(t)x(t)+\mu(t)x'(t) \\
&=-f(t)\mu(t) x(t)+ f(t)\mu(t) x(t)+\mu(t)b(t,x(t))\\
&=\mu(t)b(t,x(t)).
\end{align*}
Integrate from $t_n$ to $t$, we have
\begin{equation*}
\mu(t)x(t)-\mu(t_n)x(t_n)=\int_{t_n}^{t}\mu(s)b(s,x(s))\td s.
\end{equation*}
Then we divide the above equation by $\mu(t)$ and note $m_f(t,s)=\mu(s)/\mu(t)$. Setting $t=t_{n-1}$ yields \eqref{09}.
\end{proof}

\subsection{Numerical scheme}
We now consider the second-order sampler \eqref{sampler}. In the scheme \eqref{sampler}, we define the auxiliary point
\begin{equation}\label{03}
p:=\tx_n^*+\tfrac{h}{2}k_1=m_f(t_{n-\frac{1}{2}},t_n)\tx_n +\frac{h}{2}b(t_n,m_f(t_{n-\frac{1}{2}},t_n)\tx_n).
\end{equation}
A simple calculation shows that \eqref{sampler} can be rewritten as
\begin{equation}\label{10}
\tx_{n-1}=m_f(t_{n-1},t_n)\tx_n + hm_f\left(t_{n-1},t_{n-\frac12}\right) b\left(t_{n-\frac12},p\right).
\end{equation}
We denote this numerical flow in \eqref{10} by 
\begin{equation}\label{27}
\tx_{n-1}=\Psi_h(t_n,\tx_n).
\end{equation}
The rest of the section compares $\Psi_h$ with the exact flow in Lemma \ref{lemma:01}.

\subsection{Local truncation error}\label{seC_3_3}
Now we estimate the local truncation error at time $t_{n-1}$ by assuming that the numerical solution at time $t_n$ is accurate. Therefore, in this subsection, we always take $\tx_n = x(t_n)$ in the definition of $p$. 

We begin with the usual notion of local truncation error, defined by applying one numerical step to the exact solution at time $t_n$. The local truncation error is 
\begin{equation*}
\tau_{n-1}:=x(t_{n-1})-\Psi_h(t_n,x(t_n)).
\end{equation*}

To estimate $\tau_{n-1}$, letting 
\begin{equation*}
G(s):=m_f(t_{n-1},s)b(s,x(s)), \quad s\in[t_{n-1},t_n],
\end{equation*}
which represents the integrand in the variation of constants formula \eqref{09}. The next lemma expresses the local error as the sum of a quadrature error for the midpoint rule applied to $G$ and a consistency term coming from the approximation of the midpoint value.
\begin{lemma} \label{lemma:02}
Under Assumption \ref{ass:01},
\begin{equation*}
\tau_{n-1}=\Big(\int_{t_n}^{t_{n-1}} G(s)\td s - hG(t_{n-\frac12})\Big) + hm_f\left(t_{n-1},t_{n-\frac12}\right)\cdot\Big(b(t_{n-\frac12},x(t_{n-\frac12}))-b(t_{n-\frac12},p)\Big),
\end{equation*}
where in the definition of $p$ we takes $\tx_n = x(t_n)$.
\end{lemma}

\begin{proof}
By lemma \ref{lemma:01},
\begin{align*}
x(t_{n-1})=m_f(t_{n-1},t_n)x(t_n)+\int_{t_n}^{t_{n-1}}G(s)\td s.
\end{align*}
Subtracting the numerical evolution \eqref{10} with $\tx_n = x(t_n)$ gives 
\begin{equation*}
\tau_{n-1} =x(t_{n-1})-\tx_{n-1}
= \int_{t_n}^{t_{n-1}} G(s)\td s
- hm_f\left(t_{n-1},t_{n-\frac12}\right) b\left(t_{n-\frac12},p\right).
\end{equation*}
Adding and subtracting 
\begin{align*}
hm_f(t_{n-1},t_{n-\frac12})b\left(t_{n-\frac12},x(t_{n-\frac12})\right)
\end{align*}
yields the claimed expression.
\end{proof}

The first term in Lemma \ref{lemma:02} is the classical midpoint quadrature error, which is of order $O(h^3)$ under our smoothness assumptions.
\begin{lemma} \label{lemma:03}
Under Assumption \ref{ass:01}, the function $G(s)$ is twice continuously differentiable on $[t_{n-1},t_n]$, and there exists a constant $C>0$, independent of $h$, such that 
\begin{align*}
\left\|\int_{t_n}^{t_{n-1}} G(s)\td s - hG(t_{n-\frac12})\right\| \le Ch^3.
\end{align*}
\end{lemma}
\begin{proof}
Under Assumption \ref{ass:01}, $G\in C^2$ in $[t_{n-1},t_n]$. We expand $G(s)$ in a Taylor series around $s=t_{n-\frac{1}{2}}$
\begin{equation*}
G(s)
=G(t_{n-\frac{1}{2}})+G'(t_{n-\frac{1}{2}})(s-t_{n-\frac{1}{2}})+\frac{1}{2}G''(t_{n-\frac{1}{2}})(s-t_{n-\frac{1}{2}})^2+O(s^3).
\end{equation*}
Integrate $G(s)$ from $t_n$ to $t_{n-1}$ gives the standard midpoint error bound
\begin{align*}
\int_{t_n}^{t_{n-1}} G(s)\td s
&=\int_{t_n}^{t_{n-1}} G(t_{n-\frac{1}{2}})\td s
+\int_{t_n}^{t_{n-1}}G'(t_{n-\frac{1}{2}})(s-t_{n-\frac{1}{2}})\td s\\
&\quad+\int_{t_n}^{t_{n-1}}G''(t_{n-\frac{1}{2}})(s-t_{n-\frac{1}{2}})^2\td s+O(h^3)\\
&=hG(t_{n-\frac{1}{2}})+O(h^3),
\end{align*}
which yields the stated inequality.
\end{proof}

The second term in Lemma \ref{lemma:02} measures how well the auxiliary point $p$ approximates the exact midpoint $x(t_{n-\frac{1}{2}})$. The next lemma shows that this approximation error is of order $O(h^2)$.
\begin{lemma}\label{lemma:04}
Under Assumption \ref{ass:01} there exists a constant $C>0$, independent of $h$, such that
\begin{align*}
\|p-x(t_{n-\frac{1}{2}})\|\le Ch^2,
\end{align*}
where in the definition of $p$ we take $\tx_n = x(t_n)$.
\end{lemma}

\begin{proof}
Let $\tilde p:=m_f(t_{n-\frac12},t_n)x(t_n)$, so that $\tilde p$ is the value at $t_{n-\frac{1}{2}}$ obtained by evolving the linear ODE $x'(t)=f(t)x(t)$ backward from $x(t_n)$, while 
\begin{equation}\label{29}
m_f(t_{n-\frac12},t_n)=\exp \left(\int_{t_n}^{t_{n-\frac12}} f(\tau)\td\tau  \right)
\leq\exp \left(\|f\|_\infty\cdot\frac{h}{2}  \right)
=1+O(h),
\end{equation}
and thus $\tilde p=x(t_n)+O(h)$. By Lemma \ref{lemma:01} with $t=t_{n-\frac{1}{2}}$,
\begin{align}\label{25}
x(t_{n-\frac12})
=\tilde p+\int_{t_n}^{t_{n-\frac12}}
m_f\left(t_{n-\frac12},s\right) b\left(s,x(s)\right)\td s.
\end{align}
The boundedness of $f$, $g$ and $q$ leads to the boundedness of $m_f$ and $b$. Since the integral is over an interval of length $\frac{h}{2}$, the above boundedness implies that the integral is $O(h)$. A Taylor expansion of the integral around $s=t_n$, using the Lipschitz continuity of $b$ in $x$, shows that
\begin{align}\label{26}
\int_{t_n}^{t_{n-\frac{1}{2}}}m_f(t_{n-\frac{1}{2}},s)b(s,x(s))\td s=\frac{h}{2}b(t_n,\tilde p)+ O(h^2).
\end{align}
On the other hand, since $p=\tilde p+\frac{h}{2}b(t_n,\tilde p)$, combining \eqref{25} and \eqref{26} gives $x(t_{n-\frac{1}{2}})=p+O(h^2)$, which proves the claim.
\end{proof}

We can now combine Lemmas \ref{lemma:02}-\ref{lemma:04} with the Lipschitz property of $b$ to obtain a clean bound on the local truncation error.

\begin{lemma}\label{lemma:05}
Under Assumption \ref{ass:01}, there exists a constant $C>0$, independent of $h$, such that 
\begin{align}\label{11}
\|\tau_{n-1}\|\le Ch^3.
\end{align}
\end{lemma}

\begin{proof}
By Lemma \ref{lemma:04}, using the Lipschitz continuity of $b$ in $x$, we obtain
\begin{align*}
&hm_f\left(t_{n-1},t_{n-\frac12}\right)\Big(
b(t_{n-\frac12},x(t_{n-\frac12}))-b(t_{n-\frac12},p)\Big)\\
&\le h|m_f(t_{n-1},t_{n-\frac12})|  L \|p-x(t_{n-\frac12})\| \\
&= O(h)\cdot O(1)\cdot L\cdot O(h^2)
= O(h^3). 
\end{align*}
Then, using Lemma \ref{lemma:02}, we obtain \eqref{11}.
\end{proof}
Thus, the sampler is locally third order in time.

\subsection{Global error}
To turn the local error into a global error, we need a stability estimate for the numerical map $\Psi_h$ \eqref{10}. Informally, the next lemma shows that one step of the scheme amplifies perturbations in the initial value by at most a factor $1+O(h)$.

\begin{lemma}\label{lemma:06}
Under Assumption \ref{ass:01}, there exists a constant $C>0$, independent of $h$, such that for all $\xi,\eta \in \mathbb{R}^d$,
\begin{align}\label{12}
\|\Psi_h(t_n,\xi)-\Psi_h(t_n,\eta)\|\le(1+C h)\|\xi-\eta\|.
\end{align}
\end{lemma}

\begin{proof}
Using \eqref{10} and the triangle inequality,
\begin{align}\label{30}
\|\Psi_h(t_n,\xi)-\Psi_h(t_n,\eta)\| &= \Big\| m_f(t_{n-1},t_n)(\xi-\eta)
+hm_f(t_{n-1},t_{n-\frac12})
(b(t_{n-\frac12},p_1)-b(t_{n-\frac12},p_2))\Big\|\notag \\
&\le |m_f(t_{n-1},t_n)|\|\xi-\eta\|
+  Lh|m_f(t_{n-1},t_{n-\frac12})| \|p_1-p_2\|,
\end{align}
where $p_1$ and $p_2$ are defined by \eqref{03} with $\tx_n=\xi$ and $\tx_n=\eta$, respectively. By \eqref{29},
\begin{equation}\label{31}
\|p_1-p_2\|\le |m_f(t_{n-\frac{1}{2}},t_n)|\|\xi-\eta\| 
+ \frac{h}{2}L|m_f(t_{n-\frac{1}{2}},t_n)|\|\xi-\eta\| \le (1+C h)\|\xi-\eta\|
\end{equation}
for some $C>0$. And similar to \eqref{29}, we have
\begin{align}
&m_f(t_{n-1},t_n)=1+O(h),\label{32}\\
&m_f(t_{n-1},t_{n-\frac12})=1+O(h)\label{33}.
\end{align}

Substituting \eqref{31}-\eqref{33} into \eqref{30} and absorbing higher-order terms into the constant $C$ yield \eqref{12}.
\end{proof}

We now combine the local error bound and the stability estimate to obtain a global convergence result. Let $e_n:=x(t_n)-\tx_n$ denote the global error at time $t_n$.

\begin{theorem}\label{thm:01}
Suppose Assumption \ref{ass:01} holds, and let the numerical solution $\{\tx_n\},n=0,1,\dots,T$ be generated by \eqref{10} with terminal condition $\tx_T=x(1)$. Then there exists a constant $C>0$, independent of $h$, such that
\begin{align}\label{13}
\max_{0\le n\le T}\|x(t_n)-\tx_n\|\le Ch^2.
\end{align}
\end{theorem}

\begin{proof}
Using the definition in \eqref{27}, we write
\begin{equation*}
\|e_{n-1}\|
= \|x(t_{n-1})-\Psi_h(t_n,\tx_n)\|
\le \|x(t_{n-1})-\Psi_h(t_n,x(t_n))\| 
 + \|\Psi_h(t_n,x(t_n))-\Psi_h(t_n,\tx_n)\|. 
\end{equation*}
Applying Lemma \ref{lemma:05} to the first term and Lemma \ref{lemma:06} to the second term gives 
\begin{equation}\label{34}
\|e_{n-1}\|\le C h^3 + (1+C h)\|e_n\|.
\end{equation}
Noting that $\|e_T\|=\|x(1)-\tilde{x}_T\|=0$, applying \eqref{34} recursively for $n=T,T-1,\dots$ leads to
\begin{equation*}
\|e_{T-k}\|\le Ckh^3+O(h^4),\quad\forall k=1,\dots,T.
\end{equation*}
Since $kh\leq1$, from above we have $\max_{0\le n\le T}\|e_n\|\le Ch^2$, which is precisely \eqref{13}.
\end{proof}
Theorem \ref{thm:01} shows that the numerical scheme in the sampler \eqref{sampler} achieves global second-order accuracy, in the sense that the deterministic PF-ODE solution is approximated with error $O(h^2)$ uniformly over the entire sampling trajectory.

\section{Total variation error of the sampler}\label{sec:tv}
Building on the local and global accuracy results obtained in Section~IV, we now turn to their distributional consequences and quantify how the numerical scheme affects the total variation distance between the samples generated by the sampler and the target data distribution.

For two probability densities $p,q$ on $\mathbb{R}^{d}$, the total variation distance is defined by
\begin{equation*}
\mathrm{TV}(p,q)
:= \frac12 \int_{\mathbb{R}^{d}} |p(x)-q(x)|\td x
= \frac12\|p-q\|_{L^{1}(\mathbb{R}^{d})}.
\end{equation*}
We therefore focus on estimating $\|p-q\|_{L^{1}}$. Throughout this section, $\nabla$ denotes the gradient with respect to the spatial variable $x$, $\nabla\cdot$ denotes the divergence, and $\Delta$ denotes the Laplacian. We use $\|\cdot\|_{\mathrm{op}}$ to stand for the operator norm of the matrix with respect to the Euclidean norm.

\subsection{Assumptions}
We begin by listing all the analytic and probabilistic assumptions needed for this section. This makes it transparent which constants the final TV bounds depend on.

Let $q(\cdot,t)$ denote the exact marginal density at time $t$, and
\begin{equation*}
s(x,t) := \nabla_{x}\log q(x,t)
\end{equation*}
the corresponding exact score. The learned score network is denoted by $s_{\theta}(x,t)$. Moreover, let
\begin{equation}
 R_T := 2\sqrt{\log T}, \quad \Omega_T := \bigl\{x\in\mathbb{R}^{d}:\|x\|\le R_T\sqrt{d}\bigr\}. \label{39}
\end{equation}
Let $\tilde q_{h}(\cdot,t)$ be the marginal density at time $t$ of the numerical sampler with step size $h=1/T$, and 
\begin{align*}
\tilde s(x,t):=\nabla_{x}\log \tilde q_h(x,t)
\end{align*}
is the score of the numerical density.

\begin{assumption}
\label{ass:02}
\begin{enumerate}
\item The scalar functions $f,g \in C^{2}([0,1])$. We denote
\begin{equation*}
g^{2}(t) \le G_{\infty}, \quad \forall t\in[0,1],
\end{equation*}
for some finite constant $G_{\infty}>0$.

\item $s_{\theta}(x,t)$, $s(x,t)$ and $\tilde s(x,t)$ are twice continuously differentiable in $(x,t)$ and there exist constants $S_0,S_1,L_1>0$ such that
\begin{align}\label{35}
&\|s_{\theta}(x,t)\| \le S_{0} + S_{1}\|x\|, \|J s_{\theta}(x,t)\|_{\mathrm{op}} \le L_{1},\notag\\
&\|s(x,t)\| \le S_{0} + S_{1}\|x\|,\|J s(x,t)\|_{\mathrm{op}} \le L_{1},\notag\\
&\|\tilde s(x,t)\| \le S_{0} + S_{1}\|x\|
\end{align}
for all $x\in\mathbb{R}^{d}$, $t\in[0,1]$, where $J s_{\theta}(x,t)$ and $Js(x,t)$ are the Jacobian of $s_{\theta}$ and $s(x,t)$ with respect to $x$. 

\item Both the exact and numerical distributions have uniformly sub-Gaussian tails, in the sense that
\begin{equation}\label{42}
\mathbb{P}_{\xi\sim q(\cdot,t)}\bigl(\xi\notin \Omega_T\bigr)
+ \mathbb{P}_{\xi\sim\tilde q_{h}(\cdot,t)}\bigl(\xi\notin\Omega_T\bigr)
\le T^{-4}
\end{equation}
for all $t\in[0,1]$. And for each integer $k\in\{1,2,3,4\}$ there exists $C_1<\infty$ such that for all $t\in[0,1]$,
\begin{equation}\label{43}
\mathbb{E}_{\xi\sim q(\cdot,t)}\bigl[\|\xi\|^k\bigr] \leq C_1d^{k/2},\quad
\mathbb{E}_{\xi\sim \tilde q_{h}(\cdot,t)}\bigl[\|\xi\|^k]\leq C_1d^{k/2}.
\end{equation}

\item 
There exists a constant $C_2>0$ such that for all $t\in[0,1]$,
\begin{align}
&\Bigl(\mathbb{E}_{\xi\sim q(\cdot,t)}\bigl[\|s(\xi,t)\|^{2}\mathbf{1}_{\Omega_T}(\xi)\bigr]\Bigr)^{1/2}
\le C_2\sqrt{d},\label{40}\\
&\Bigl(\mathbb{E}_{\xi\sim \tilde q(\cdot,t)}\bigl[\|\tilde s(\xi,t)\|^{2}\mathbf{1}_{\Omega_T}(\xi)\bigr]\Bigr)^{1/2}
\le C_2\sqrt{d},\label{41}
\end{align}
where $\mathbf{1}_{\Omega_T}(\cdot)$ denotes the characteristic function of the set $\Omega_T$.
\end{enumerate}
\end{assumption}
Assumption \ref{ass:02} collects mild growth, smoothness and tail conditions required for the TV bounds analysis. Such ``regularity" and ``tail'' assumptions are standard in non-asymptotic diffusion analyses \cite{Li2023,Li2024_1}, and weaker than imposing global structural conditions such as log-concavity that are used in some Wasserstein-based results \cite{Gao2024,Huang2025}.

Under Assumption~\ref{ass:02} we will repeatedly use the following error metrics for the score network:
\begin{equation*}
\varepsilon_{\mathrm{score}}:= \sqrt{\sup_{t\in[0,1]}\mathbb{E}_{\xi\sim q(\cdot,t)}
\bigl[\|s_{\theta}(\xi,t) - s(\xi,t)\|^{2}\bigr]},
\end{equation*}
\begin{equation*}
\varepsilon_{\mathrm{Jac}}
:= \sup_{t\in[0,1]}
\mathbb{E}_{\xi\sim q(\cdot,t)}
\bigl[\|J s_{\theta}(\xi,t) - J s(\xi,t)\|_{\mathrm{op}}\bigr].
\end{equation*}
The first quantity controls the score mismatch in mean square; the second controls the mismatch of the Jacobian, which will be crucial when estimating divergences of error fields. Note that \eqref{35} in Assumption~\ref{ass:02} guarantees that the Radon-Nikodym derivative between $\tilde q_{h}$ and $q$ is uniformly bounded. Therefore, there exists some constant $C_3>0$
\begin{equation}\label{36}
\mathbb{E}_{\xi\sim \tilde q_h(\cdot,t)}[\zeta(\xi)]\le C_3\mathbb{E}_{\xi\sim q(\cdot,t)}[\zeta(\xi)],
\end{equation}
where $\zeta$ is any nonnegative function such that these expectations exist.

\subsection{Continuity equations and renormalization }
The deterministic probability flow ODE with exact score $x'(t) = f(t)x(t) - \tfrac12 g^{2}(t)s(x(t),t)$ has the corresponding continuity equation
\begin{equation}\label{14}
\partial_{t} q(x,t)
+ \nabla\cdot\bigl(q(x,t)v(x,t)\bigr) = 0, 
\end{equation}
where the exact velocity field is
\begin{equation*}
v(x,t) := f(t)x - \tfrac12 g^{2}(t)s(x,t).
\end{equation*}

On the numerical side, we use the sampler scheme \eqref{sampler} from Section \ref{seC_2} with the learned score $s_{\theta}$.
Then we have the corresponding continuity equation
\begin{equation}\label{16}
\partial_{t} \tilde q_{h}(x,t)
+ \nabla\cdot\bigl(\tilde q_{h}(x,t)\tilde v_{h}(x,t)\bigr) = 0,
\end{equation}
where $\tilde v_{h}(x,t)$ is the modified vector field. Note that the modified-equation analysis of the second-order Strang splitting shows that a single step of the deterministic sampler can be interpreted as the time-$h$ flow of the modified vector field
\begin{equation}\label{37}
\tilde v_{h}(x,t)
= f(t)x - \tfrac12 g^{2}(t)s_{\theta}(x,t)
+ O(h^{2}).
\end{equation}
Inspired by \eqref{37}, we define
\begin{equation*}
r(x,t) = \frac{1}{h^2}\left(\tilde v_{h}(x,t)-f(t)x+\frac{1}{2} g^{2}(t)s_{\theta}(x,t)\right).
\end{equation*}
Then \eqref{37} leads to
\begin{equation}\label{15}
\tilde v_{h}(x,t)= f(t)x - \tfrac12 g^{2}(t)s_{\theta}(x,t)+ h^{2} r(x,t),
\end{equation}
where $r(x,t)$ should be a function of $O(1)$.  

By construction, both flows start from the same initial value at $t=1$, i.e., $q(\cdot,1) = \tilde q_{h}(\cdot,1)$. We now introduce the difference of densities
\begin{equation*}
w(x,t) := \tilde q_{h}(x,t) - q(x,t).
\end{equation*}
Subtracting \eqref{14} from \eqref{16}, we obtain the linear continuity equation for $w$ with an explicit source term, i.e.,
\begin{equation}\label{17}
\partial_{t} w(x,t)
+ \nabla\cdot\bigl(w(x,t)v(x,t)\bigr)
= \nabla\cdot F(x,t),
\end{equation}
where
\begin{equation*}
F(x,t) :=
\tfrac12 g^{2}(t)\tilde q_{h}(x,t)\bigl(s_{\theta}(x,t)-s(x,t)\bigr) 
- h^{2}\tilde q_{h}(x,t)r(x,t)
\end{equation*}
is a vector field. Equation \eqref{17} gathers all deviations between the exact and numerical drifts: the first part comes from score mismatch, the second from the modified vector field. We remark that \eqref{17} is important for the analysis: it expresses how the density error is transported along the exact velocity field and driven by an explicit divergence source $\nabla\cdot F$. The next lemma turns this PDE into a differential inequality for the $L^{1}$-norm of $w$, which is exactly what we need to control the TV distance.

\begin{lemma}
\label{lemma:07}
Under Assumption~\ref{ass:02}, the solution $w$ of \eqref{17} satisfies
\begin{align}\label{18}
\frac{\td }{\td t}\|w(\cdot,t)\|_{L^{1}(\mathbb{R}^{d})}
\le \int_{\mathbb{R}^{d}} \bigl|\nabla\cdot F(x,t)\bigr|\td x   
\end{align}
for almost every $t\in[0,1]$.
\end{lemma}

\begin{proof}
The proof uses a standard renormalization argument for transport equations. Let $\varphi_{\varepsilon}:\mathbb{R}\to\mathbb{R}$ be a smooth convex function such that $\varphi_{\varepsilon}(z) \longrightarrow |z|$ a.e. as $\varepsilon\downarrow0$, $|\varphi_{\varepsilon}'(z)|\le 1\ $ for all $z\in\mathbb{R}$ and $|\varphi_{\varepsilon}(z)|\le |z|$ a.e..
Multiplying \eqref{17} by $\varphi_{\varepsilon}'(w)$ and using the chain rule, we get
\begin{equation*}
\partial_{t}\varphi_{\varepsilon}(w)
+ \nabla\cdot\bigl(\varphi_{\varepsilon}(w)v\bigr)
+ \bigl(\varphi_{\varepsilon}'(w)w - \varphi_{\varepsilon}(w)\bigr)\nabla\cdot v
= \varphi_{\varepsilon}'(w)\nabla\cdot F.
\end{equation*}
Let $B_R=\{x\in\mathbb{R}^d:\|x\|\leq R\}$, then integrating over $B_R$ and using the divergence theorem, we obtain
\begin{equation*}
\frac{\td }{\td t}\int_{B_R} \varphi_{\varepsilon}(w)\td x 
= \int_{B_R}
\left((\varphi_{\varepsilon}(w)-\varphi_{\varepsilon}'(w)w )
\nabla\cdot v
+ \varphi_{\varepsilon}'(w)\nabla\cdot F \right)
\td x 
 + \mathcal{B}_{\varepsilon},
\end{equation*}
where the boundary term $\mathcal{B}_{\varepsilon}=\int_{B_R}\varphi_{\varepsilon}(w)v\cdot n\td S$. The sub-Gaussian tails and at-most-linear growth of $v$ in Assumption~\ref{ass:02} guarantees that $\mathcal{B}_{\varepsilon}\to 0$ as the radius $R$ tends to infinity. When letting $R\to \infty$, the sphere $B_R\to\mathbb{R}^d$ and leads to
\begin{equation}\label{28}
\frac{\td }{\td t}\int_{\mathbb{R}^d} \varphi_{\varepsilon}(w)\td x 
= \int_{\mathbb{R}^d}
\left((\varphi_{\varepsilon}(w)-\varphi_{\varepsilon}'(w)w )
\nabla\cdot v
+ \varphi_{\varepsilon}'(w)\nabla\cdot F \right)
\td x .
\end{equation}
The coefficient $\nabla\cdot v$ is bounded, and for $\varphi_{\varepsilon}$ one has $\varphi_{\varepsilon}'(z)z - \varphi_{\varepsilon}(z)\to 0$ a.e. as $\varepsilon\downarrow 0$. By the a.e. convergence $\varphi_{\varepsilon}(z)\to |z|$ and the dominated convergence theorem, we take the limit on both sides of \eqref{28}
\begin{equation*}
\frac{\td }{\td t}\lim_{\varepsilon \to0}\int_{\mathbb{R}^{d}} \varphi_{\varepsilon}(w)\td x
=\frac{\td }{\td t}\int_{\mathbb{R}^{d}} |w(x,t)|\td x
\le \int_{\mathbb{R}^{d}} \bigl|\nabla\cdot F(x,t)\bigr|\td x,
\end{equation*}
which is exactly \eqref{18}.
\end{proof}

Lemma \ref{lemma:07} is the bridge between the PDE \eqref{17} and a scalar differential inequality for $\|w(\cdot,t)\|_{L^{1}}$. The rest of the section is devoted to estimating the right-hand side in \eqref{18} in terms of the score error, the Jacobian error, and the time step $h$.

\subsection{Structure and bounds of the source term}

The source field $F$ combines two qualitatively different contributions:
\begin{itemize}
\item the \emph{model error}, driven by the difference $s_{\theta} - s$;
\item the \emph{discretization error}, driven by the correction field $r(x,t)$ and scaled by $h^{2}$.
\end{itemize}
We therefore split
\begin{equation*}
\nabla\cdot F
= \tfrac12 g^{2}\nabla\cdot\bigl(\tilde q_{h}(s_{\theta}-s)\bigr)
- h^{2}\nabla\cdot(\tilde q_{h} r).
\end{equation*}
Correspondingly, we define the non-negative quantities
\begin{equation*}
M(t) := \int_{\mathbb{R}^{d}}
\Bigl|
\nabla\cdot\bigl(\tilde q_{h}(s_{\theta}-s)\bigr)
\Bigr| \td x,\quad
D(t) := \int_{\mathbb{R}^{d}}
\bigl|\nabla\cdot(\tilde q_{h} r)\bigr| \td x,
\end{equation*}
so that, using $|g^{2}(t)|\le G_{\infty}$ in Assumption \ref{ass:02}, we have
\begin{equation}\label{19}
\int_{\mathbb{R}^{d}} |\nabla\cdot F(x,t)|\td x
\le \tfrac12 G_{\infty} M(t) + h^{2} D(t).
\end{equation}

We next bound $M(t)$ and $D(t)$ using Assumption~\ref{ass:02} and the correction field $r$.

\begin{lemma}
\label{lemma:08}
Under Assumption~\ref{ass:02}, there exists a constant $C_{r}>0$ depending on $r$ such that 
\begin{equation*}
\|r(x,t)\|
\le C_{r}\bigl(1 + \|x\|\bigr),\quad
\bigl|\nabla\cdot r(x,t)\bigr|
\le C_{r}\bigl(1 + \|x\|\bigr),
\end{equation*}
for all $x\in\mathbb{R}^{d}$ and $t\in[0,1]$.
\end{lemma}

\begin{proof}
By the symmetry of Strang splitting, the one-step numerical flow map of the ODE \eqref{04} can be written as
\begin{equation*}
\Phi_h=e^{\frac{h}{2}A}e^{h\mathcal{L}_B}e^{\frac{h}{2}A},
\end{equation*}
where $\mathcal{L}_B$ is the Lie derivative operator defined as $(\mathcal L_B\phi)(x)=B(x)\cdot\nabla\phi(x)$ with $\phi(x)$ being a function. Classical Baker--Campbell--Hausdorff expansions give
\begin{equation}\label{38}
\Phi_{h}
= \exp \left(
h(A+B) + h^{3}R + O(h^{5})
\right),
\end{equation}
where $R$ is a linear combination of double commutators such as $[A,[A,B]]$ and $[B,[A,B]]$. Interpreting \eqref{38} as the time-$h$ flow of a modified vector field $\tilde F_{h} = F + h^{2}r + O(h^{4})$, we see that $r$ is itself a linear combination of Lie brackets of $A$ and $B$ and their derivatives.

The coefficients $f,g,s_{\theta}$ entering $A$ and $B$ satisfy linear growth and bounded-derivative assumptions by Assumption~\ref{ass:02}. Lie brackets of such vector fields inherit at-most-linear growth in $x$ and bounded divergence. This yields the stated bounds with a constant $C_{r}$ depending only on the constants in Assumption~\ref{ass:02}.
\end{proof}

Lemma~\ref{lemma:08} makes precise the idea that the correction field $r$ behaves like a well-behaved drift field: it grows at most linearly in $\|x\|$, and its divergence stays under control. This property is critical because it allows the factor $h^2$ in front of $r$ to translate directly into an $O(h^2)$ contribution to the total variation distance.

We now estimate the model-error contribution $M(t)$ in terms of $\varepsilon_{\mathrm{Jac}}$ and $\varepsilon_{\mathrm{score}}$.

\begin{lemma}
\label{lemma:09}
Under Assumption~\ref{ass:02} it holds that
\begin{equation}\label{20}
M(t)
\le
C_{M}\Bigl(d\varepsilon_{\mathrm{Jac}}
+ \sqrt{d}\varepsilon_{\mathrm{score}}\Bigr)
\end{equation}
for all $t\in[0,1]$, where 
\begin{align*}
C_M=\max\{C_3,\sqrt{C_2+2\sqrt{2S_0^4+2S_1^4C_1}}C_3^{1/2}\}.
\end{align*}
\end{lemma}

\begin{proof}
We expand the divergence as
\begin{equation*}
\nabla\cdot\bigl(\tilde q_{h}(s_{\theta}-s)\bigr)
= \tilde q_{h}\nabla\cdot(s_{\theta}-s)
+ \nabla \tilde q_{h} (s_{\theta}-s).
\end{equation*}
Taking absolute values and integrating leads to
\begin{equation*}
M(t) \le M_{1}(t) + M_{2}(t),
\end{equation*}
where
\begin{equation*}
M_{1}(t):= \int \tilde q_{h}(x,t)\bigl|\nabla\cdot(s_{\theta}-s)(x,t)\bigr|\td x,\quad
M_{2}(t):= \int \bigl|\nabla \tilde q_{h}(x,t) (s_{\theta}-s)(x,t)\bigr|\td x.
\end{equation*}

For $M_{1}(t)$, we use the inequality
\begin{equation*}
\bigl|\nabla\cdot(s_{\theta}-s)\bigr|
= \bigl|\mathrm{tr}(J s_{\theta} - J s)\bigr|
\le d\|J s_{\theta} - J s\|_{\mathrm{op}}.
\end{equation*}
Hence
\begin{equation*}
M_{1}(t)
\le d
\mathbb{E}_{\xi\sim \tilde q_h(\cdot,t)}
\bigl[
\|J s_{\theta}(\xi,t) - J s(\xi,t)\|_{\mathrm{op}}
\bigr].
\end{equation*}
On the set $\Omega_T$ defined by \eqref{39}, \eqref{36} leads to
\begin{equation*}
M_{1}(t)
\le dC_3
\mathbb{E}_{\xi\sim  q(\cdot,t)}
\bigl[
\|J s_{\theta}(\xi,t) - J s(\xi,t)\|_{\mathrm{op}}
\bigr].
\end{equation*}
Using the definition of $\varepsilon_{\mathrm{Jac}}$, we obtain $M_{1}(t) \le C_3d\varepsilon_{\mathrm{Jac}}$.

For $M_{2}(t)$, we note that $\nabla \tilde q_{h} = \tilde q_{h}\tilde s$. Thus
\begin{equation*}
M_{2}(t)
= \mathbb{E}_{\xi\sim \tilde q_h(\cdot,t)}
\bigl[\tilde s(\xi,t)\cdot (s_{\theta}(\xi,t)-s(\xi,t))\bigr].
\end{equation*}
Applying Cauchy--Schwarz and using the change-of-measure factor between $\tilde q_{h}$ and $q$ on $\Omega_T$, we get
\begin{align*}
M_{2}(t)
&\le \Bigl(
\mathbb{E}_{\xi\sim \tilde q_h}[\|\tilde s\|^{2}]
\Bigr)^{1/2}
\Bigl(
\mathbb{E}_{\xi\sim \tilde q_h}[\|s_{\theta}-s\|^{2}]
\Bigr)^{1/2}\\
&\le C_3^{1/2}
\Bigl(
\mathbb{E}_{\xi\sim \tilde q_h}[\|\tilde s\|^{2}]
\Bigr)^{1/2}
\varepsilon_{\mathrm{score}}.
\end{align*}
Combined with the sub-Gaussian moment bounds \eqref{43} on $\tilde q_{h}$ and the inequality \eqref{41} involving $C_2$, this yields
\begin{align}\label{44}
\left(
\mathbb{E}_{\xi\sim \tilde q_h}[\|\tilde s\|^{2}]
\right)^{\frac{1}{2}}&\leq
\left( \mathbb{E}_{\xi\sim \tilde q_h}[\|\tilde s\|^{2}\mathbf{1}_{\Omega_T}(\xi)]+\mathbb{E}_{\xi\sim \tilde q_h}[\|\tilde s\|^{2}\mathbf{1}_{\Omega_T^c}(\xi)]\right)^{\frac{1}{2}}\notag\\
&\leq \left(C_2d+\left(\mathbb{E}_{\xi\sim \tilde q_h}[\|\tilde s\|^{4}] \cdot\mathbb{E}_{\xi\sim \tilde q_h}[\mathbf{1}_{\Omega_T}^2(\xi)] \right)^{\frac{1}{2}}\right)^{\frac{1}{2}} \notag
\\
&\leq \left(C_2d+\left(8S_0^4+8S_1^4\mathbb{E}_{\xi\sim \tilde q_h}[\|x\|^4]\right)^{\frac{1}{2}}T^{-2}\right)^{\frac{1}{2}}\notag\\
&\leq \left(C_2d+\left(2\sqrt{2S_0^4+2S_1^4C_1}d\right)T^{-2}\right)^{\frac{1}{2}} \notag\\
&\leq \sqrt{C_2+2\sqrt{2S_0^4+2S_1^4C_1}} \sqrt{d}.
\end{align}
Hence
\begin{equation*}
M_{2}(t)
\le \sqrt{C_2+2\sqrt{2S_0^4+2S_1^4C_1}}C_3^{1/2}\sqrt{d}\varepsilon_{\mathrm{score}}.
\end{equation*}

Combining the bounds for $M_{1}(t)$ and $M_{2}(t)$ and absorbing constants into $C_{M}$, we obtain \eqref{20}.
\end{proof}

Lemma~\ref{lemma:09} makes explicit how the TV error will depend on two qualitatively different aspects of the score network: the mismatch of its Jacobian and the mismatch of its values. In particular, the appearance of $\varepsilon_{\mathrm{Jac}}$ is distinctive of this PDE-based approach and does not arise in simple coupling arguments.

We now turn to the discretization source.

\begin{lemma}
\label{lemma:10}
Under Assumption~\ref{ass:02} it holds that
\begin{equation}\label{21}
D(t)
\le C_{D}d(1+R_T)
\end{equation}
for all $t\in[0,1]$, where
\begin{align*}
C_D=\left(C_r+\sqrt{C_1}\right)\sqrt{C_2+2\sqrt{2S_0^4+2S_1^4C_1}}.
\end{align*}
\end{lemma}

\begin{proof}
Note that
\begin{align*}
D(t)
&=\int |\tilde q_{h}\nabla\cdot r+ \nabla \tilde q_{h} \cdot r|\td x\\
&\le \int \tilde q_{h}(x,t)
|\nabla\cdot r(x,t)|\td x
 + \int |\nabla \tilde q_{h}(x,t)\cdot r(x,t)|\td x=: D_{1}(t) + D_{2}(t).
\end{align*}

For $D_{1}(t)$, Lemma~\ref{lemma:08} gives $|\nabla\cdot r(x,t)| \le C_{r}(1+\|x\|)$. Using the sub-Gaussian moment bounds \eqref{43} and the fact that $\|x\|\le R_T\sqrt{d}\le \sqrt{d}(1+R_T)$ with high probability, by a derivation similar to that of equation \eqref{44}, we obtain
\begin{equation*}
D_{1}(t)
\le C_r\bigl(1+\mathbb{E}_{\xi\sim \tilde q_h(\cdot,t)}[\|\xi\|]\bigr)
\le (C_r+\sqrt{C_1})d^{1/2}(1+R_T).
\end{equation*}
For $D_{2}(t)$, we again use $\nabla \tilde q_{h} = \tilde q_{h}\tilde s$, so that
\begin{align*}
D_{2}(t)
&= \mathbb{E}_{\xi\sim \tilde q_h(\cdot,t)}
\bigl[|\tilde s(\xi,t)\cdot r(\xi,t)|\bigr]\\
&\le
\Bigl(\mathbb{E}_{\xi\sim \tilde q_h}[\|\tilde s\|^{2}]\Bigr)^{1/2}
\Bigl(\mathbb{E}_{\xi\sim \tilde q_h}[\|r\|^{2}]\Bigr)^{1/2}.
\end{align*}
The same reasoning as in \eqref{44} yields
\begin{equation*}
\Bigl(\mathbb{E}_{\xi\sim \tilde q_h}[\|\tilde s\|^{2}]\Bigr)^{1/2}
\le \sqrt{C_2+2\sqrt{2S_0^4+2S_1^4C_1}}\sqrt{d}.
\end{equation*}
Similarly, using Lemma~\ref{lemma:08} and the moment bounds \eqref{43} on $\tilde q_{h}$,
\begin{equation*}
\Bigl(\mathbb{E}_{\xi\sim \tilde q_h}[\|r\|^{2}]\Bigr)^{1/2}
\le (C_r+\sqrt{C_1})\sqrt{d}(1+R_T).
\end{equation*}
Multiplying these two estimates we obtain a bound of order $d(1+R_T)$ for $D_{2}(t)$. Combining with the estimate for $D_{1}(t)$ and renaming constants yields \eqref{21}.
\end{proof}

Taken together, Lemmas~\ref{lemma:09} and~\ref{lemma:10} show that the source term in the renormalized error equation is uniformly integrable in time, with a clean separation between model error and discretization error contributions.

\subsection{Main TV error bounds}

We now assemble the ingredients. Recall that $w(\cdot,t) = \tilde q_{h}(\cdot,t) - q(\cdot,t)$ and that $w(\cdot,1)\equiv 0$ because both flows share the same initial value at $t=1$. Integrating the inequality \eqref{18} from Lemma~\ref{lemma:07} backward in time and using \eqref{19}, we obtain
\begin{equation}\label{22}
\|w(\cdot,0)\|_{L^{1}}
\le \int_{0}^{1}
\Bigl(
\tfrac12 G_{\infty} M(t)
+ h^{2} D(t)
\Bigr)\td t.
\end{equation}

We first state a general bound that displays the three contributions: Jacobian error, score error, and discretization error.

\begin{theorem}
\label{theorem:02}
Under Assumption~\ref{ass:02} it holds that
\begin{equation}\label{23}
\mathrm{TV}\bigl(\tilde q_{h}(\cdot,0),q(\cdot,0)\bigr)
\le C_4\Bigl(d\varepsilon_{\mathrm{Jac}}
+ \sqrt{d}\varepsilon_{\mathrm{score}}
+ d\tfrac{1+2\sqrt{\log T}}{T^{2}}
\Bigr),
\end{equation}
for all $t\in[0,1]$, where $C_4=\max\{\tfrac12 G_{\infty}C_M,C_D\}$.
\end{theorem}

\begin{proof}
Combining Lemmas~\ref{lemma:09} and~\ref{lemma:10} with \eqref{19}, we obtain
\begin{equation*}
\tfrac12 G_{\infty} M(t) + h^{2}D(t)
\le \tfrac12 G_{\infty}C_M\Bigl(
d\varepsilon_{\mathrm{Jac}}
+ \sqrt{d}\varepsilon_{\mathrm{score}}
\Bigr)
+ C_D h^{2} d (1+R_T).
\end{equation*}
The three terms on the right-hand side are independent of $t$. Inserting this bound into \eqref{22} and integrating over $t\in[0,1]$ yield
\begin{equation*}
\|w(\cdot,0)\|_{L^{1}}
\le \tfrac12 G_{\infty}C_M\Bigl(
d\varepsilon_{\mathrm{Jac}}
+ \sqrt{d}\varepsilon_{\mathrm{score}}\Bigr) + C_D h^{2} d\bigl(1+2\sqrt{\log T}\bigr).
\end{equation*}
Since
\begin{equation}
\mathrm{TV}\bigl(\tilde q_{h}(\cdot,0),q(\cdot,0)\bigr)
= \tfrac12 \|w(\cdot,0)\|_{L^{1}},
\end{equation}
we obtain \eqref{23} after absorbing constants into a single constant $C_4$.
\end{proof}

If focusing on the term involving the number of sampling steps $T$, Theorem \ref{theorem:02} shows that the TV distance between the target and generated distributions is upper bounded by $O(d/T^2)$. Compared to existing theoretical results, one of the key contributions of Theorem \ref{theorem:02} is the reduction of the dependence on the dimension $d$ to a linear term. In many previous works, the convergence rate or error bounds of high-order samplers exhibit a quadratic (or even higher-order) dependence on $d$ \cite{Li2024_3,Li2025}.

Theorem~\ref{theorem:02} disentangles three analytically distinct sources of TV error: the mismatch in the score, the mismatch in its Jacobian, and the numerical discretization of the PF-ODE. In particular, the Jacobian term $d\varepsilon_{\mathrm{Jac}}$ appears naturally from the renormalized PDE argument and does not arise if one only tracks pointwise score errors.

As a clean corollary, when the score is exact, the TV error is purely driven by the numerical discretization of the probability flow.

\begin{corollary}
\label{corollary:01}
Under Assumption~\ref{ass:02}, suppose in addition that the learned score is exact, i.e.\ $s_{\theta} \equiv s$. Then it holds that
\begin{equation}\label{24}
\mathrm{TV}\bigl(\tilde q_{h}(\cdot,0),q(\cdot,0)\bigr)
\le
C_Dd\frac{1+2\sqrt{\log T}}{T^{2}}.
\end{equation}
\end{corollary}

\begin{proof}
If $s_{\theta}\equiv s$, then $\varepsilon_{\mathrm{score}}=0,\varepsilon_{\mathrm{Jac}}=0$. The proof is completed by \eqref{23}. 
\end{proof}
Theorem \ref{theorem:02} and Corollary \ref{corollary:01} show that the second-order Strang-splitting sampler is not only globally second order in the trajectory-wise sense of Section~IV, but also achieves a genuinely second-order convergence rate in total variation distance to the target data distribution, up to a mild logarithmic factor in the number of steps. Moreover, the TV bound decomposes in a structurally transparent way into a model-error part (controlled by $\varepsilon_{\mathrm{score}}$ and $\varepsilon_{\mathrm{Jac}}$) and a pure-discretization part of order $T^{-2}$. This decomposition clarifies how numerical analysis, score learning, and statistical error interact in deterministic diffusion samplers, and it provides a concrete target for improving both the network and the integrator.

\section{Experiments}
\label{sec:experiments}
In this section, we conduct numerical experiments to validate the theory presented above. Specifically, we validate the TV error bounds given in Section~V using an artificial two-dimensional diffusion model example. In this example, the score and TV distance can be computed explicitly, allowing us to elucidate their dependence on the step size $h$. 

\subsection{Experimental setup}
We consider a synthetic two-dimensional Gaussian data distribution $\mathcal{N}(\mu,\Sigma)$ on $\mathbb{R}^2$ with density $q(x,0)$, where $\mu=(1,-1)\in\mathbb{R}^2$ and $\Sigma=\left[ \begin{matrix}
1.5 & 0.6  \\
0.6 & 0.8  \\
\end{matrix} \right]\in\mathbb{R}^{2\times 2}$ is positive definite. The forward noising process follows the variance-preserving (VP) construction in \eqref{forward}: for $t\in[0,1]$, $x_t  \sim \mathcal{N}\bigl(\alpha(t)x_0, \sigma^2(t) I\bigr)$.

We use the usual VP parameterization \cite{Ho2020} with
$\alpha^2(t) + \sigma^2(t) = 1$ and a linear noise schedule
$\beta(t) = \beta_0 + (\beta_1-\beta_0)t$ with $(\beta_0,\beta_1) =
(0.1,20.0)$. We note that $\alpha(t)=\exp\bigl(\frac{1}{2}\int_0^t\beta(s)\td s\bigr)$. The drift $f$ and diffusion $g$ in the PF-ODE are obtained from $\alpha$ and $\sigma$ as in Section~II. 

During the sampling process for the diffusion model, we consider two alternative approaches to handling the score. The first approach is to compute the exact score by $s^\ast(x,t)=\nabla_x\log q(x,t)$, where $q(x,t)$ is a known Gaussian density. In this case, the PF-ODE is accurate, and the error arises solely from the discretization of the dynamics, i.e., the sampler. The second approach is to approximate the score using a fully connected neural network $\varepsilon_\theta(x_t,t)$ to predict the noise, then the score can be replaced by $-\frac{\varepsilon_\theta(x_t,t)}{\sigma(t)}$, just as the diffusion models in practical applications. In this case, the error stems from both the sampler and the training of the neural network.

We train a set of fully connected neural networks with varying architectures, using 1, 2, 3 or 4 hidden layers and widths of 100, 200, 400 or 800. We generate $5\times10^4$ i.i.d.\ training triples $(x_0,t,\xi)$ with $x_0\sim q(\cdot,0)$, $t\sim \mathrm{Unif}[0,1]$, and $\xi\sim \mathcal{N}(0,I)$. The noisy input is computed by $x_t = \alpha(t)x_0 + \sigma(t)\xi$, and the network is trained with the standard noise-prediction loss
\[
\mathcal{L}(\theta)
= \mathbb{E}_{x_0,t,\xi}
\bigl\|\varepsilon_\theta(x_t,t) - \xi\bigr\|^2.
\]
In the optimization, we use the Adam optimizer with an initial learning rate $10^{-5}$ and an exponential decay to $10^{-6}$ over the course of training. The number of iterations is set as $1.5\times 10^4$. Note that the MSE training loss for the noise-prediction network cannot be close to zero, even for arbitrarily large learner networks. The reason is that the regression target is the standard Gaussian noise
$\xi \sim \mathcal{N}(0,I)$ in the forward process
$x_t = \alpha(t) x_0 + \sigma(t)\xi$.  Conditioned on $(x_t,t)$, the
noise $\xi$ is still random. Note that $(x_t,\xi)$ is jointly Gaussian and
$$
\mathrm{Var}(\xi | x_t,t)= I - \sigma^2(t)\bigl(\alpha^2(t)\Sigma + \sigma^2(t) I\bigr)^{-1}.
$$
Therefore, the optimal predictor is
$\xi^\star(x_t,t) = \mathbb{E}[\xi | x_t,t]$ and the corresponding training loss has an expected minimum $\mathcal{L}_\ast= \mathbb{E}_{t}\Bigl[\frac{1}{d}\mathrm{tr}\mathrm{Var}(\xi | x_t,t)\Bigr]$. Plugging the quantities and the covariance matrix $\Sigma$ of our two-dimensional Gaussian data into the expression above and evaluating the resulting one-dimensional integral over
$t\in[0,1]$ yields $\mathcal{L}_\ast \approx 0.2705$.

The values reported in table \ref{table01} are empirical losses obtained by training on $5\times10^4$ points, and they are close to $\mathcal{L}_\ast$, showing that the learner networks approximate the optimal predictor with acceptable errors.
\begin{table}[!t]
\caption{training loss}
\label{table01}
\centering
\small
\setlength{\tabcolsep}{5pt}
\renewcommand{\arraystretch}{1.1}

\begin{tabular}{c|cccc}
\hline
\multirow{2}{*}{\begin{tabular}[c]{@{}c@{}}Hidden\\ layers\end{tabular}} & \multicolumn{4}{c}{Width} \\ \cline{2-5}
& 100 & 200 & 400 & 800 \\
\hline
1 & 3.922e-01 & 3.099e-01 & 2.939e-01 & 2.884e-01 \\
2 & 2.920e-01 & 2.766e-01 & 2.714e-01 & 2.686e-01 \\
3 & 2.783e-01 & 2.719e-01 & 2.688e-01 & 2.647e-01 \\
4 & 2.759e-01 & 2.703e-01 & 2.667e-01 & 2.610e-01 \\
\hline
\end{tabular}
\end{table}

At the stage of generating samples, we take $\tx_T\sim\mathcal{N}(0,I)$ and integrate the PF-ODE backward from $t=1$ to $t=0$ with the second-order sampler of Section~III-B.1. We use either the exact score or the learned score by the above neural networks in the sampler. We perform sampling with different numbers of time steps $T$. For each choice of $T$, we generate $2 \times10^4$ points in $\mathbb{R}^2$. Note that 
\[
q(x,0)= \frac{1}{2\pi\sqrt{|\Sigma|}}\exp\bigl(-(x-\mu)^\top\Sigma^{-1} (x-\mu)\bigr)
\]
is explicit. For the sampler distribution, we fit a Gaussian kernel density estimator $\hat{p}_T$ to the generated samples. We then estimate
\[
\mathrm{TV}\bigl(\hat{p}_T, q(\cdot,0)\bigr)
= \tfrac12 \int_{\mathbb{R}^2}
\bigl|\hat{p}_T(x)-q(x,0)\bigr|\mathrm{d}x
\]
by Monte Carlo integration. Therefore, the TV distance reported in our experiments also reflects the Monte Carlo integration error.

\subsection{Numerical results}
The resulting log-log plots of TV distance versus the step size $h=1/T$ for various $h$ are shown in Figure \ref{fig:tv_distance}. Figure \ref{fig:a} and \ref{fig:b} present plots of the exact and learned scores, respectively, obtained with a neural network of 2 hidden layers and a width of 200. From Figure \ref{fig:a}, the points lie almost exactly on the reference line, and a least squares fit of $\log\mathrm{TV}$ against $\log h$ yields an empirical slope very close to~$2$. This confirms that, in the absence of training error, the sampler achieves the $O(h^2)$ TV convergence rate predicted by our analysis. From Figure \ref{fig:b}, the error bound of Theorem~2 applies: the TV distance splits into an $h$-independent training-error term, controlled by the score and Jacobian errors $(\varepsilon_{\mathrm{score}}, \varepsilon_{\mathrm{Jac}})$, plus an $O(h^2)$ discretization term. For the range of step sizes considered, the curve is nearly parallel to the slope-$2$ line, indicating that the $O(h^2)$ discretization error remains the dominant contribution. Compared with using the exact score, the curve is shifted upward, reflecting the additional TV error caused by the imperfect score; at the smallest step size, we observe a slight flattening of the curve that is consistent with an emerging training-error floor. We remark that neural networks of other sizes produce plots that are almost identical.

Recall that in Section \ref{sec:tv}, Theorem \ref{theorem:02} and Corollary \ref{corollary:01} imply that the TV distance between the sampler and the data distribution is bounded by $O(1/T^2)$; namely, the distance has a second-order convergence with respect to $h$. Thus, our results are consistent with the theory.

\begin{figure}[!t]
\centering
\subfloat[]{\label{fig:a}\includegraphics[width=0.6\linewidth]{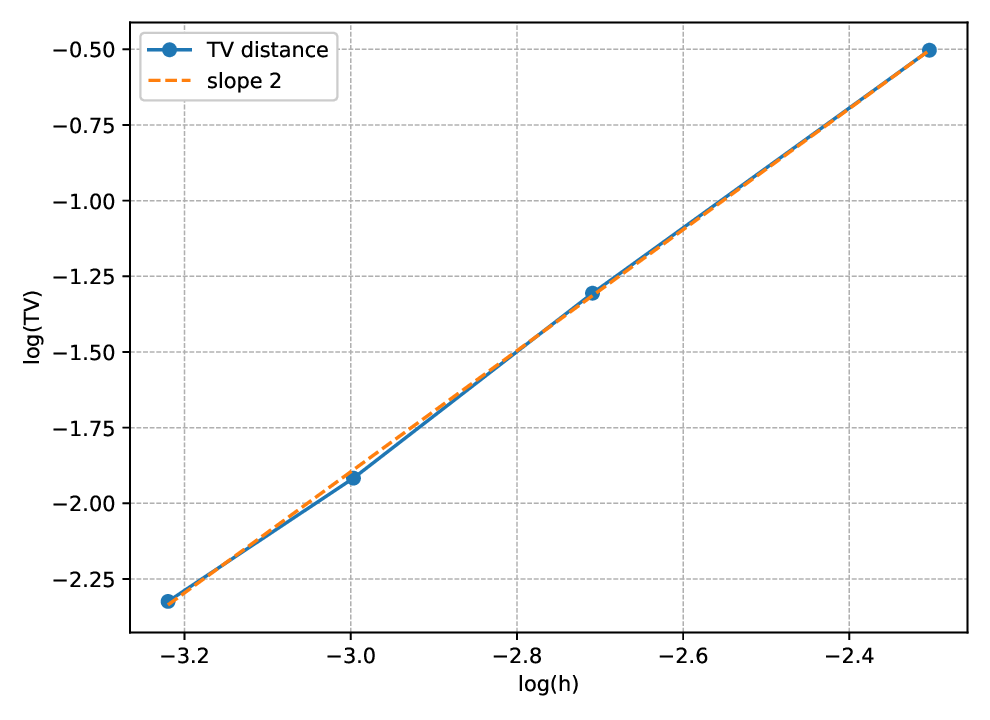}}
\\
\subfloat[]{\label{fig:b}\includegraphics[width=0.6\linewidth]{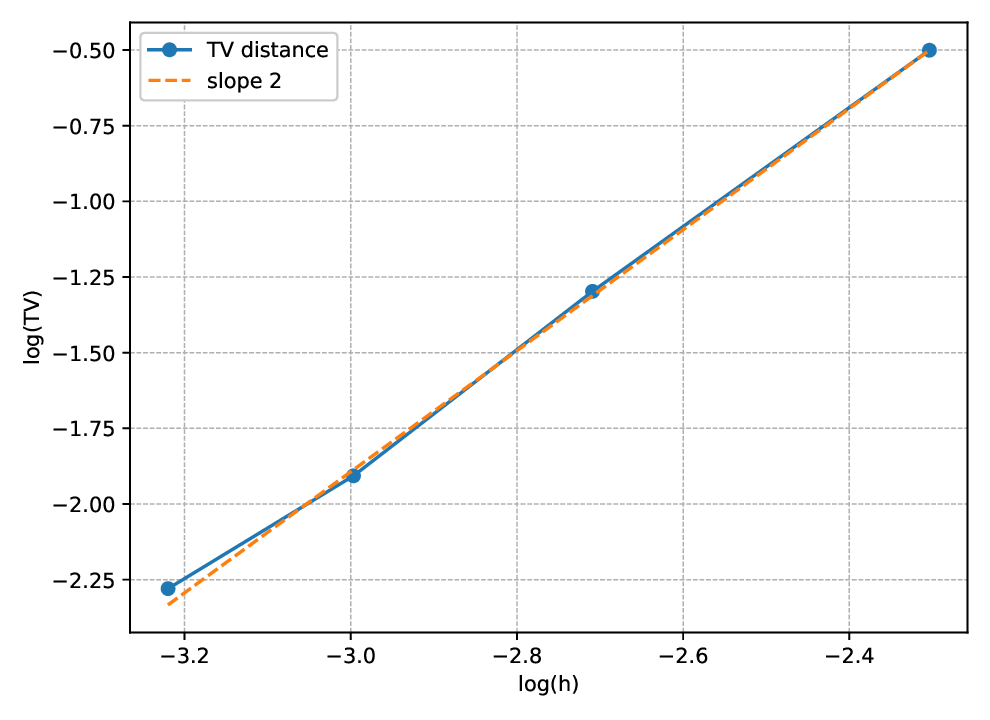}}
\caption{Log-log plots of the total variation distance versus the step size $h$ for the
Strang--Midpoint sampler: (a) Results using exact score; (b) Results using learned score (a neural network with 2 hidden layers and width 200).}
\label{fig:tv_distance}
\end{figure}

\subsection{Discussion}
The experiments provide a clean sanity check of our theoretical TV
bounds. When the score is exact, the second-order Strang–Midpoint sampler
exhibits the expected $O(h^2)$ convergence in TV distance, verifying
Theorem~3. With a realistic finite-capacity neural network, the TV
error behaves as the sum of an $O(h^2)$ discretization term and an
$h$-independent training-error term, in line with Theorem~2.
This example thus already captures the qualitative interaction between score learning and numerical integration in probability-flow ODE samplers and supports the relevance of our analysis for high-dimensional diffusion models.

\section{Conclusion}
We have proposed and analysed a simple operator–splitting sampler for probability flow ODEs in the diffusion model. By separating the PF-ODE into a linear variance-preserving part and a nonlinear score-driven part, and composing their flows with a Strang splitting combined with a midpoint Runge–Kutta step, the method achieves global second-order accuracy at the trajectory level under standard Lipschitz and moment assumptions. On top of this path-wise analysis, we derived non-asymptotic bounds in TV distance between the samples and the target data distribution. These bounds isolate the contributions from time discretization, score value error and score Jacobian error, and exhibit a genuinely second-order dependence on the step size without imposing strong structural assumptions on the data distribution or the neural network.

We complemented the theory with two-dimensional experiments, where both the exact score and the ground-truth TV distance can be computed. With the exact score, the empirical TV error decays at a rate very close to $O(h^2)$, in line with the theoretical prediction. When the score is approximated by fully connected neural networks of varying depth and width, the TV curves are shifted upward but preserve an essentially quadratic dependence on the step size until they reach a training-induced floor that matches the theoretically computed minimal loss. These results suggest that the Strang–Midpoint sampler offers a practical and theoretically transparent compromise between simplicity and accuracy, and that the error decomposition developed in this paper can serve as a useful guideline for balancing network training and numerical design in diffusion generative modeling.

\bibliographystyle{unsrt}   
\bibliography{refs}         

\end{document}